\documentclass[11pt,a4paper,final]{article}
\usepackage{amsmath}
\usepackage{amssymb}
\usepackage{amsxtra}
\usepackage{amsthm}
\usepackage{bbm}			
\usepackage{cite}

\usepackage[left=3.1cm,right=3.1cm,top=3.3cm,bottom=3.3cm]{geometry}

\usepackage[usenames,dvipsnames]{xcolor}
\usepackage[plainpages=false,pdfpagelabels,colorlinks=true,linkcolor=blue,citecolor=Magenta]{hyperref} 

\usepackage{showkeys}
\usepackage[latin1]{inputenc}       
\usepackage[T1]{fontenc}        
\usepackage{lmodern}            
\usepackage{graphicx}
\usepackage{subfig}		
\usepackage{booktabs}		
\usepackage{multirow}
\usepackage{enumitem}		
\usepackage{listings} 	
\usepackage{mathtools}  
\usepackage{floatrow}		
\usepackage{caption}

\usepackage{colortbl}
\definecolor{dunkelgrau}{rgb}{0.8,0.8,0.8}
\definecolor{hellgrau}{rgb}{0.9,0.9,0.9}

\newcommand{\filt}{\ensuremath{\mathfrak{F}}}

\newcommand{\R}{\ensuremath{\mathbb{R}}}

\newcommand{\E}{\ensuremath{\mathbb{E}}}
\renewcommand{\P}{\ensuremath{\mathbb{P}}}

\newcommand{\oR}{\ensuremath{\overline{\mathbb{R}}}}

\newcommand{\ones}{\ensuremath{\mathbbm{1}}}

\renewcommand{\>}{\right\rangle}
\newcommand{\<}{\left\langle}
\renewcommand{\d}{\ensuremath{\,\text{d}}}
\newcommand{\id}{\ensuremath{\text{Id}}}

\newcommand{\bx}{\ensuremath{\overline{x}}}

\newcommand{\bv}{\ensuremath{\overline{v}}}

\newcommand{\bp}{\ensuremath{\overline{p}}}

\renewcommand{\bv}{\ensuremath{\overline{v}}}

\newcommand{\expreturn}{\ensuremath{\mu^*}}

\newcommand{\g}{\ensuremath{\mathcal{G}}}
\newcommand{\h}{\ensuremath{\mathcal{H}}}

\newcommand{\q}{\ensuremath{\mathcal{Q}}}
\newcommand{\X}{\ensuremath{\mathcal{X}}}

\newcommand{\proj}{\ensuremath{\mathcal{P}}}

\renewcommand{\Box}{\ensuremath{\mbox{\small$\,\square\,$}}}

\theoremstyle{plain}
\newtheorem{theorem}{Theorem}[section]

\newtheorem{proposition}[theorem]{Proposition}

\theoremstyle{definition}
\newtheorem{remark}{Remark}[section]

\newtheorem{definition}{Definition}[section]

\newtheorem{algorithm}{Algorithm}[section]
\newtheorem{assumption}{Assumption}[section]

\DeclareMathOperator*\dom{dom}%
\DeclareMathOperator*\argmin{arg\,min}%
\DeclareMathOperator*\Prox{Prox}%
\DeclareMathOperator*\cvar{CVaR}%
\DeclareMathOperator*\varisk{VaR}%
\DeclareMathOperator*\essinf{essinf}%
\DeclareMathOperator*\esssup{essup}%

\numberwithin{equation}{section}  

\title{Convex risk minimization via proximal splitting methods}

\author{Radu Ioan Bo\c t
\thanks{Department of Mathematics, Chemnitz University of Technology, D-09107 Chemnitz, Germany, e-mail: radu.bot@mathematik.tu-chemnitz.de. Research partially supported by DFG (German Research Foundation), project BO 2516/4-1.}
\and Christopher Hendrich
\thanks{Department of Mathematics, Chemnitz University of Technology, D-09107 Chemnitz, Germany, e-mail: christopher.hendrich@mathematik.tu-chemnitz.de. Research supported by a Graduate Fellowship of the Free State Saxony, Germany.}}
\date{\today}

\begin{document}
\maketitle

{\bf Abstract.} In this paper we investigate the applicability of a recently introduced primal-dual splitting method in the context of solving portfolio optimization problems which assume the minimization of risk measures associated to different convex utility functions. We show that, due to the splitting characteristic of the used primal-dual method, the main effort in implementing it constitutes in the calculation of the proximal points of the utility functions, which assume explicit expressions in a number of cases. When quantifying risk via the meanwhile classical Conditional Value-at-Risk, an alternative approach relying on the use of its dual representation is presented as well. The theoretical results are finally illustrated via some numerical experiments on real and synthetic data sets.



\section{Introduction and preliminaries}\label{sectionIntro}

In financial mathematics, quantifying the risk of future random outcomes is an essential concern for decision makers who, of course, have their own attitude towards risk. In the classical portfolio theory by Markowitz, the variance was used to measure the risk of future returns. However, subsequent developments in the theory of risk measures showed that the variance does not have desirable properties, one major drawback being given by the fact that it measures deviations of the random variable in both directions, i.\,e. it penalizes losses and gains in the same way.

When dealing with uncertainty, modern risk measures build on the wide-spread recognition that asymmetry is a desirable property since investors have different positions on downside and upside outcomes. The first axiomatic way of defining risk measures has been given by Artzner, Delbaen, Eber and Heath in \cite{Artetal99} and refers to \textit{coherent risk measures}. Nevertheless, it has become a standard in modern risk management to assess the riskiness of a portfolio by means of \textit{convex risk measures}, introduced by F\"ollmer and Schied in \cite{FoeSch}, as well as of \textit{convex deviation measures}, introduced by Rockafellar, Uryasev and Zabarankin in \cite{RUZ06}. The convex deviation measures are connected with the risk measures when those are applied to the difference between a random variable and its expectation, instead of to the random variable itself, both being designed to be used in risk analysis. However, deviation measures evaluate uncertainty in the form of nonconstancy, whereas risk measures can be understood as estimates of capital requirements for future net worths.

In this paper we give a unifying framework for solving portfolio optimization problems assuming the minimization of a general convex risk functional subject to  constraints on the expected return of the portfolio and on the budget. The convex risk measure in the objective is expressed in terms of the Optimized Certainty Equivalent (OCE), a fundamental concept for quantifying risk by means of a utility function introduced by Ben-Tal and Teboulle in \cite{BenTeb86} (see also \cite{BenTeb07}). For the particular case of the Conditional Value-at-Risk an alternative approach involving its dual representation is also considered. The approach we propose in this paper assumes the solving of these constrained optimization problems, which in their majority have an nondifferentiable objective function with an intricate formulation, via a primal-dual proximal splitting method which has been recently introduced in \cite{BotHend12e}. In the last years one can notice an increasing interest
 in primal-dual algorithms when solving nondifferentiale convex optimization problems (see, for instance, \cite{BotCseHein12,BotHend12,BotHend12e,BriCom11,ChaPoc11,ComPes12,Con12,Vu11}), as they achieve a full splitting which assumes a separate evaluation of each function via its proximal points, while the occurring linear continuous operators and their adjoints are only evaluated via forward steps. Most primal-dual algorithms also allow inexact evaluations of the proximal points  which, however, can have a negative impact on the overall performance. On the other hand, in a lot of applications, as it will be also the case for majority of the convex risk measures considered in this paper, exact implementations of the proximal operators are possible. We refer to the mentioned literature for applications of the primal-dual methods in image and signal processing, location theory and machine learning.

The structure of the paper is the following. In the remaining of this subsection we give some elements of convex analysis, present the primal-dual proximal algorithm along with its convergence behaviour and introduce the necessary apparatus for defining convex risk measures. Section 2 is dedicated to the formulation of the portfolio optimization problem to be solved, when the risk is quantified via the Optimized Certainty Equivalent, and to investigations on the applicability of the primal-dual method in this context. In Section 3 an alternative approach for solving the portfolio optimization problem having as objective the Conditional Value-at-Risk is presented. We illustrate the applicability of the proposed primal-dual method in the context of portfolio optimization problems in Section 4 by some numerical experiments on real and synthetic data.

\subsection{Convex analysis}\label{subsec21}
Let $\h$ be a real Hilbert space with \textit{inner product} $\left\langle \cdot ,\cdot \right\rangle$ and associated \textit{norm} $\left\| \cdot \right\| = \sqrt{\left\langle \cdot, \cdot \right\rangle}$. The symbol $\R_{++}$ denotes the set of strictly positive real numbers and $\R_+ := \R_{++} \cup \{0\}$. For a given set $S \subseteq \h$, the function $\delta_S : \h \rightarrow \overline \R := \R \cup \{\pm \infty\}$, defined by $\delta_S(x) = 0$ for $x \in S$ and $\delta_S(x) = +\infty$, otherwise, denotes its \textit{indicator function}. For a function $f: \h \rightarrow \oR$ we denote by $\dom f := \left\{ x \in \h : f(x) < +\infty \right\}$ its \textit{effective domain} and call $f$ \textit{proper} if $\dom f \neq \varnothing$ and $f(x)>-\infty$ for all $x \in \h$. Let be
$$\Gamma(\h) := \{f: \h \rightarrow \oR: f \ \mbox{is proper, convex and lower semicontinuous}\}.$$
The \textit{conjugate function} of $f$ is $f^*:\h \rightarrow \oR$, $f^*(p)=\sup{\left\{ \left\langle p,x \right\rangle -f(x) : x\in\h \right\}}$ for all $p \in \h$ and, if $f \in \Gamma(\h)$, then $f^* \in \Gamma(\h)$, as well. The \textit{(convex) subdifferential} of $f: \h \rightarrow \oR$ at $x \in \h$ is the set $\partial f(x) = \{p \in \h : f(y) - f(x) \geq \left\langle p,y-x \right\rangle \ \forall y \in \h\}$, if $f(x) \in \R$, and is taken to be the empty set, otherwise. For a linear continuous operator $L: \h \rightarrow \g$, the operator $L^*: \g \rightarrow \h$, defined via $\< Lx,y  \> = \< x,L^*y  \>$ for all $x \in \h$ and all $y \in \g$, denotes its \textit{adjoint}.

Having two proper functions $f,\,g : \h \rightarrow \oR$, their \textit{infimal convolution} is defined by $f \Box g : \h \rightarrow \oR$, $(f \Box g) (x) = \inf_{y \in \h}\left\{ f(y) + g(x-y) \right\}$ for all $x \in \h$, being a convex function when $f$ and $g$ are convex.  The parallel sum of the subdifferentials of $f$ and $g$, seen as set-valued operators, is defined as $\partial f \Box \partial g : \h \rightrightarrows \h$, $(\partial f \Box \partial g) (x) = \{p \in \h : x \in (\partial f)^{-1}(p) + (\partial g)^{-1}(p)\}$, where $(\partial f)^{-1}(p) = \{x \in \h: p \in \partial f(x)\}$. One has that $(\partial f \Box \partial g)(x) \subseteq \partial (f \Box g)(x)$ for every $x \in \h$. For $f \in \Gamma(\h)$, its subdifferential $\partial f :  \h \rightrightarrows \h$ is a  maximally monotone operator (cf. \cite{Roc70}) and by $\Prox_{f}(x)$ we denote the \textit{proximal point} of $f$ at $x\in\h$, representing the unique optimal solution of the optimization problem
\begin{align}\label{prox-def}
\inf_{y\in \h}\left \{f(y)+\frac{1}{2}\|y-x\|^2\right\}.
\end{align}
For every $\gamma \in \R_{++}$  we have \textit{Moreau's decomposition formula} (cf. \cite[Theorem 14.3]{BauCom11})
\begin{align}
	\label{res-indentity}
	\id = \Prox\nolimits_{\gamma f} + \gamma \Prox\nolimits_{\gamma^{-1}f^*} \circ \gamma^{-1}\id,
\end{align}
where $\id$ denotes the identity operator on $\h$.  When $S \subseteq \h$ is a nonempty convex and closed set, the proximal point of $\delta_S$ at $x \in \h$ is
$$\Prox\nolimits_{\delta_S}(x) =\proj_S(x) = \argmin_{y \in S} \frac{1}{2}\|y-x\|^2,$$
being nothing else than the \textit{projection} of $x$ on $S$.

For $\h$ and $\g_i$, $i=1,...,m$, given real Hilbert spaces, $f \in \Gamma(\h)$, $g_i, l_i \in \Gamma(\g_i)$, $i=1,...,m$, and $L_i:\h \rightarrow \g_i$, $i=1,...,m,$ nonzero linear continuous operators we consider the convex optimization problem
\begin{align*}
	(P) \quad \inf_{x \in \h}{\left\{f(x)+\sum_{i=1}^m (g_i \Box l_i)(L_ix) \right\}}
\end{align*}
and its Fenchel-type conjugate dual problem (see, for instance, \cite{Bot10, BotHend12e, ComPes12})
\begin{align*}
	(D) \quad \sup_{(v_1,\ldots,v_m) \in \g_1\times\ldots\times\g_m}{\left\{-f^*\left(- \sum_{i=1}^m L_i^*v_i\right) - \sum_{i=1}^m \left( g_i^*(v_i) + l_i^*(v_i) \right) \right\} }.
\end{align*}
By denoting with $v(P)$ and $v(D)$ the optimal objective values of the problems $(P)$ and $(D)$, respectively, in general one has weak duality, i.e., $v(P) \geq v(D)$. Strong duality, which is the situation when $v(P)=v(D)$ and the dual problem $(D)$ has an optimal solution, holds, when some appropriate qualification condition is fulfilled. The following error-tolerant proximal splitting algorithm which is suitable for simultaneously solving the problems $(P)$ and $(D)$ was given in \cite[Algorithm 3.1]{BotHend12e}. Its competitiveness is emphasized by several numerical experiments in the context of location and image processing problems, also in comparison to other recently introduced iterative schemes.

\begin{algorithm}\label{risk_alg1} \text{ }\newline
Let $x_0 \in \h$, $(v_{1,0}, \ldots, v_{m,0}) \in \g_1 \times \ldots \times \g_m$ and $\tau$ and $\sigma_i$, $i=1,...,m,$ be strictly positive real numbers such that
$$\tau \sum_{i=1}^m \sigma_i \|L_i\|^2 < 4. $$
Furthermore, let $(\lambda_n)_{n\geq 0}$ be a sequence in $(0,2)$, $(a_n)_{n\geq 0}$ a sequence in $\h$, $(b_{i,n})_{n\geq 0}$ and $(d_{i,n})_{n\geq 0}$ sequences in $\g_i$ for all $i=1,\ldots,m$
and set
	\begin{align}\label{risk_A1}
	  \left(\forall n\geq 0\right) \begin{array}{l}  \left\lfloor \begin{array}{l}
		p_{1,n} = \Prox_{\tau f}\left( x_n - \frac{\tau}{2} \sum_{i=1}^m L_i^* v_{i,n} \right) + a_n \\
		w_{1,n} = 2p_{1,n} - x_n \\
		\text{For }i=1,\ldots,m  \\
				\left\lfloor \begin{array}{l}
					p_{2,i,n} = \Prox_{\sigma_i g_i^*}\left(v_{i,n} +\frac{\sigma_i}{2} L_i w_{1,n} \right) + b_{i,n} \\
					w_{2,i,n} = 2 p_{2,i,n} - v_{i,n} \\
				\end{array} \right.\\
		z_{1,n} = w_{1,n} - \frac{\tau}{2} \sum_{i=1}^m L_i^* w_{2,i,n} \\
		x_{n+1} = x_n + \lambda_n ( z_{1,n} - p_{1,n} ) \\
		\text{For }i=1,\ldots,m  \\
				\left\lfloor \begin{array}{l}
					z_{2,i,n} = \Prox_{\sigma_i l_i^*}\left(w_{2,i,n} + \frac{\sigma_i}{2}L_i (2 z_{1,n} - w_{1,n}) \right) + d_{i,n} \\
					v_{i,n+1} = v_{i,n} + \lambda_n (z_{2,i,n} - p_{2,i,n}). \\
				\end{array} \right. \\ \vspace{-4mm}
		\end{array}
		\right.
		\end{array}
	\end{align}
\end{algorithm}

\begin{remark}\label{rpartcase}
When $l=\delta_{\left\{0\right\}}$, the infimal convolution $g \Box l$ is nothing else than the function $g$. In this situation, the conjugate of $l$ is the function identical to zero and for all $\sigma \in \R_{++}$ one has
$\Prox_{\sigma l^*} = \id$.
\end{remark}

The subsequent theorem was given in \cite[Theorem 3.1]{BotHend12e} and characterizes the convergence behaviour of the sequences generated by Algorithm \ref{risk_alg1}.
\begin{theorem}\label{risk_th0}
Suppose that there exists $x \in \h$ such that
\begin{align}\label{zin1.2}
	0 \in  \partial f(x) +  \sum_{i=1}^m L_i^*((\partial g_i \Box \partial l_i)(L_ix)).
\end{align}
\begin{enumerate}[label={(\roman*)}]
	\setlength{\itemsep}{-2pt}
		\item  \label{th3.01} If
$$\sum_{n=0}^{+\infty}\lambda_n\|a_n\|_{\h} < +\infty, \quad \sum_{n=0}^{+\infty}\lambda_n (\|d_{i,n}\|_{\g_i} + \|b_{i,n}\|_{\g_i} ) < + \infty, i=1,\ldots,m,$$ and
$\sum_{n=0}^{+\infty} \lambda_n (2 - \lambda_n) = +\infty,$
then
		\begin{enumerate}[label={(\alph*)}]
	\setlength{\itemsep}{-2pt}
		\item  \label{risk_th0.1} $(x_n,v_{1,n},\ldots,v_{m,n})_{n\geq 0}$ converges weakly to a point $(\bx,\bv_1,\ldots,\bv_m)\in\h\times\g_1\times\ldots\times\g_m$ such that, when setting
		\begin{align*}
						\bp_1 &= \Prox\nolimits_{\tau f}\left( \bx - \frac{\tau}{2} \sum_{i=1}^m L_i^* \bv_i \right), \\
						\text{and }\bp_{2,i} & =\Prox\nolimits_{\sigma_i g_i^*}\left(\bv_{i} +\frac{\sigma_i}{2} L_i (2\bp_1-\bx) \right), \ i =1,...,m,
		\end{align*}
$\bp_1$ is an optimal solution to the primal problem $(P)$, $(\bp_{2,1},\ldots,\bp_{2,m})$ is an optimal solution to the dual problem $(D)$ and $v(P) = v(D)$.
		\item  \label{risk_th0.2}$\lambda_n (z_{1,n}-p_{1,n}) \rightarrow 0 \  (n \rightarrow +\infty)$ and  $\lambda_n(z_{2,i,n}-p_{2,i,n}) \rightarrow  0 \  (n \rightarrow +\infty)$ for $i=1,...,m$.
		\item  \label{risk_th0.3} whenever $\h$ and $\g_i, i=1,...,m,$ are finite-dimensional Hilbert spaces, $a_n \rightarrow 0 \  (n \rightarrow +\infty)$ and $b_{i,n} \rightarrow 0 \  (n \rightarrow +\infty)$ for $i=1,...,m$, then $(p_{1,n})_{n \geq 0}$ converges strongly to an optimal solution to $(P)$ and  $(p_{2,1,n},\ldots,p_{2,m,n})_{n \geq 0}$ converges strongly to an optimal solution to $(D)$.
		\end{enumerate}
		\item  \label{risk_th0.02}  If
$$\sum_{n=0}^{+\infty}\|a_n\|_{\h} < +\infty, \quad \sum_{n=0}^{+\infty} (\|d_{i,n}\|_{\g_i} + \|b_{i,n}\|_{g_i} ) < + \infty, i=1,\ldots,m, \quad \inf_{n\geq 0} \lambda_n > 0$$
$$\mbox{and} \ f \ \mbox{and} \ g_i^*, i=1,...,m, \ \mbox{are uniformly convex},$$
then $(p_{1,n})_{n \geq 0}$ converges strongly to an optimal solution to the primal problem $(P)$, $(p_{2,1,n},\ldots,p_{2,m,n})_{n \geq 0}$ converges strongly to an optimal solution to the dual problem $(D)$ and $v(P) = v(D)$.
\end{enumerate}
\end{theorem}

\subsection{Convex risk measures}
Let $(\Omega,\filt,\P)$ be an atomless probability space, where the elements $\omega$ of $\Omega$ represent future states, or individual scenarios (and are allowed to be only finitely many), $\filt$ is a \textit{$\sigma$-algebra} on measurable subsets of $\Omega$ and $\P$ is a \textit{probability measure} on $\filt$. For a measurable random variable $X:\Omega\to\R \cup \{+\infty\}$ the \emph{expectation
value} with respect to $\P$ is defined by $E[X]:=\int_{\Omega} X(\omega) \d\P(\omega)$. Whenever $X$ takes the value $+\infty$ on a subset of positive measure
we have $\E[X] = +\infty$.  Equalities between random variables are to be interpreted in an \textit{almost surely (a.s.)} way.  Random variables $X:\Omega\to\R \cup \{+\infty\}$ which take a constant value $\lambda\in\R,$ i.e $X=\lambda$ a.s., will be identified with the real number $\lambda$. Similarly, inequalities of the form $X \geq \lambda$, $X \leq \lambda$, $X \leq Y$, etc., are to be viewed in the sense of holding almost surely. By $F_X$ we denote the \textit{distribution function} of $X$, i.\,e. $F_X(\lambda)=\P(X\leq \lambda)$. By taking this into account, \textit{essential supremum} and \textit{essential infimum} of a random variable $X$ are, respectively,
\begin{align*}
	\esssup(X) &= \inf\left\{a \in \R: \P(X>a)=0\right\} = \inf\left\{a \in \R : X \leq a \right\} \\
	\essinf(X) &= - \esssup(-X) = \sup \left\{a \in \R : X \geq a \right\}.
\end{align*}
Each random variable $X$ can be represented as $X=X_+ - X_-$, where $X_+,X_-$ are random variables defined via $X_+(\omega)=\max \{X(\omega),0\}$ and $X_-(\omega)=\max\{-X(\omega),0\}$ for all $\omega \in \Omega$.

Consider further the real Hilbert space
$$L^2:=L^2(\Omega,\filt,\P)= \left\{ X : \Omega \rightarrow \R \cup \{+\infty\} : X \text{ is measurable, } \int_{\Omega} \left|X(\omega)\right|^2 \d\P(\omega) < + \infty \right\}$$
endowed with \textit{inner product} and \textit{norm} defined for arbitrary $X,Y \in L^2$ via
\begin{align*}
	\< X,Y \> = \int_{\Omega} X(\omega)Y(\omega) \d\P(\omega) \ \mbox{and} \ \left\|X\right\| = \left(\<X,X\>\right)^{\frac{1}{2}} = \left( \int_{\Omega} \left(X(\omega)\right)^2 \d\P(\omega) \right)^{\frac{1}{2}},
\end{align*}
respectively.

\begin{definition}[Risk functions]
A proper function $\rho: L^2 \rightarrow \oR$ is called \textit{risk function}. The risk function $\rho$ is said to be
\begin{enumerate}[label={(\roman*)}]
	\setlength{\itemsep}{-2pt}
		\item {\em convex}, if $\rho(\lambda X+(1-\lambda)Y) \leq \lambda \rho(X) + (1-\lambda)\rho(Y)$ for all $\lambda\in(0,1)$, $X,Y \in L^2$;
		\item {\em positively homogeneous}, if $\rho(0)=0$ and $\rho(\lambda X)=\lambda \rho(X)$ for all $\lambda \in \R_{++}$, $X \in L^2$;
		\item {\em monotone}, if $X \geq Y$ implies $\rho(X)\leq \rho(Y)$ for all $X,Y \in L^2$;
		\item {\em cash-invariant}, if $\rho(X+c)=\rho(X)-c$ for all $c\in\R$, $X\in L^2$;
		\item a {\em convex risk measure}, if $\rho$ is convex, monotone and cash-invariant;
		\item a {\em coherent risk measure}, if $\rho$ is a positively homogeneous convex risk measure.
\end{enumerate}
\end{definition}
Axioms for coherent risk measures were first given in the literature by Artzner, Delbaen, Eber and Heath in \cite{Artetal99}, while later one, F\"ollmer and Schied considered in  \cite{FoeSch} the convex risk measures, by replacing the sublinearity with the weaker assumption of convexity. More precisely, when the value $\rho(X)$ is understood as a capital requirement for the future net worth $X$, a convex risk measure guarantees that the capital requirement of the convex combination of two positions does not exceed the convex combination of the capital requirements of the positions taken separately. For properties and examples of coherent and convex risk measures we refer to \cite{Artetal99, BenTeb07, BotFra11, FoeSch02, FoeSch02b, LueDoe05, RocUry00, RocUry02, RocUryZab02, RUZ06}.

In our investigations a central role will be played by a generalized convex risk measure associated to the so-called Optimized Certainty Equivalent, which was introduced for concave utility functions in \cite{BenTeb86} and adapted to convex utility functions in \cite{BotFra11}. For the utility functions considered throughout this paper we make the following assumption.

\begin{assumption}[Convex utility function] \label{risk_as1}
Let $u:\R \rightarrow \oR$ be a proper, convex, lower semicontinuous and nonincreasing function such that $u(0)=0$ and $-1 \in \partial u(0)$.
\end{assumption}

In the literature the two conditions imposed on $u$ are known as the \textit{normalization conditions} and are equivalent to $u(0)=0$ and $u(t)\geq -t$ for all $t\in\R$. The generalized convex risk measure we use in order to quantify the risk was given under the name Optimized Certainty Equivalent (OCE) in \cite{BenTeb07} and is defined as (see, also, \cite{BotFra11})
\begin{align}\label{risk_d1}
	\rho_u: L^2 \rightarrow \R \cup \{+\infty\}, \ \rho_u(X) = \inf_{\lambda\in\R}\left\{\lambda + \E\left[u(X+\lambda)\right] \right\}.
\end{align}
By Assumption \ref{risk_as1}, it follows that $\rho_u(X) \geq - \E\left[X\right]$ for every $X\in L^2$ and that $\rho_u$ fulfills the requirements of being a convex risk measure.

\section{Solving a general portfolio optimization problem}\label{risk_sectionApproach}

Consider a portfolio with a number of $N \geq 1$ different positions with returns $R_i \in L^2$, $i=1,\ldots,N$, a nonzero vector of expected returns $\mu=(\E\left[R_1\right],\ldots, \E\left[R_N\right])^T$ and $\mu^* \leq \max_{i=1,...,N} \E\left[R_i\right]$ a given lower bound for the expected return of the portfolio. In this section we discuss the employment of Algorithm \ref{risk_alg1} when solving for different convex utility functions the optimization problem
\begin{align}
\label{risk_p2}
\inf_{\substack{x^T\mu\geq \expreturn,\ x^T\ones^N =1,\\x=(x_1,...,x_N)^T \in \R^N_+}} \rho_u\left(\sum_{i=1}^N x_iR_i\right),
\end{align}
which assumes the minimization of the risk of the portfolio subject to constraints on the expected return of the portfolio and on the budget. Here, $\ones^N$ denotes the vector in $\R^N$ having all entries equal to $1$. By using \eqref{risk_d1}, we obtain the following reformulation of the problem \eqref{risk_p2}
\begin{align}
\label{risk_p3}
\inf_{\substack{x^T\mu\geq \expreturn,\ x^T\ones^N =1,\\x=(x_1,...,x_N)^T \in \R^N_+,\ \lambda \in \R}} \left\{\lambda + \E\left[ u\left(\sum_{i=1}^N x_iR_i+\lambda\right)\right] \right\},
\end{align}
which will prove to be more suitable for being solved by means of the primal-dual proximal splitting algorithm presented in the previous section. In this sense, the following result, which relates the optimal solutions of the two optimization problems is of certain importance.
\begin{proposition}\label{risk_th4}
The following statements are true.
\begin{itemize}
\item[(a)] If $(\bx, \overline \lambda)$ is an optimal solution to \eqref{risk_p3}, for $\bx = (\bx_1,\ldots, \bx_N)^T$, then $\bx$ is an optimal solution to \eqref{risk_p2}.

\item[(b)] If $\bx = (\bx_1,\ldots, \bx_N)^T$ is an optimal solution to \eqref{risk_p2} and
$$\overline \lambda \in \argmin_{\lambda \in \R} \left \{\lambda + \E\left[ u\left(\sum_{i=1}^N \bar x_iR_i+\lambda\right)\right]\right\},$$
then $(\bx, \overline \lambda)$ is an optimal solution to \eqref{risk_p3}.
\end{itemize}
\end{proposition}
\begin{proof}
Denote by $\X=\left\{x\in\R^N_+: x^T\mu\geq \expreturn, \ x^T\ones^N =1 \right\}$.
\begin{itemize}
\item[(a)] Since $(\bx,\overline{\lambda})\in\X\times\R$ is an optimal solution to \eqref{risk_p3}, we have for every $(x,\lambda) \in \X \times \R$
\begin{align*}
	\lambda + \E\left[ u\left(\sum_{i=1}^N x_iR_i+\lambda\right)\right] \geq \overline{\lambda} + \E\left[ u\left(\sum_{i=1}^N \bx_iR_i+\overline{\lambda}\right)\right] \geq \rho_u\left(\sum_{i=1}^N \bx_iR_i\right).
\end{align*}
Passing to the infimum over $\lambda \in \R$ yields
\begin{align*}
	\rho_u\left(\sum_{i=1}^N x_iR_i\right) \geq \rho_u\left(\sum_{i=1}^N \bx_iR_i\right) \ \forall x\in\X,
\end{align*}
hence, $\bx \in \X$ is an optimal solution to \eqref{risk_p2}.

\item[(b)] The conclusion follows by noticing that for every $(x,\lambda) \in \X \times \R$ we have
\begin{eqnarray*}
	\lambda + \E\left[ u\left(\sum_{i=1}^N x_iR_i+\lambda\right)\right] \geq \rho_u\left(\sum_{i=1}^N x_iR_i\right) & \geq & \rho_u\left(\sum_{i=1}^N \bx_iR_i\right)\\
& = &  \overline{\lambda} + \E\left[ u\left(\sum_{i=1}^N \bx_iR_i+\overline{\lambda}\right)\right].
\end{eqnarray*}
\end{itemize}
\end{proof}

\begin{remark}\label{risk_recfunc}
A sufficient condition guaranteeing that
$$\argmin_{\lambda \in \R} \left \{\lambda + \E\left[ u\left(X+\lambda\right)\right]\right\} \neq \varnothing \ \forall X \in L^2$$
was given in \cite[Theorem 4]{BotFra11} and reads
\begin{equation}\label{risk_eqrec}
\{d \in \R: u_\infty(d) = - d\} = \{0\},
\end{equation}
where $u_\infty : \R \rightarrow \overline \R$, $u_\infty(d) = \sup\{u(x+d) - u(x) : x \in \dom u\}$, denotes the \textit{recession function} of the function $u$. Moreover, in the light of the same result, it follows that under \eqref{risk_eqrec}
$$\rho_u(X) = \sup_{\substack{\Xi \in L^2\\ \mathbb{E}(\Xi)=-1}} \left \{\langle X, \Xi \rangle -\mathbb{E}\left[u^{*}(\Xi)\right] \right \} \ \forall X \in L^2,$$
thus $\rho_u$ is lower semicontinuous. Since $\X$ is compact, this further implies that  \eqref{risk_p2} has an optimal solution and, consequently, that
\eqref{risk_p3} has an optimal solution, too.  All particular convex utility functions we deal with in this paper fulfill condition \eqref{risk_eqrec}.
\end{remark}

According to Proposition \ref{risk_th4}, determining an optimal solution to problem \eqref{risk_p3} will lead to an optimal solution to the portfolio optimization problem \eqref{risk_p2}. However, as we will show in the following, problem \eqref{risk_p3} is a particular case of the problem $(P)$, thus it can be solved by Algorithm \ref{risk_alg1}, but also by some other primal-dual proximal splitting methods. In order to show this, let us first consider
the linear (hence continuous) operator
$$K:\R^N \times \R \rightarrow L^2,\ (x_1,\ldots,x_n,\lambda) \mapsto \sum_{i=1}^N x_iR_i + \lambda.$$
In order to determine its adjoint operator $K^*: L^2 \rightarrow \R^N \times \R$ we use that
\begin{align*}
	\< K(x,\lambda),Z \> = \int_{\Omega} \left(\sum_{i=1}^N x_iR_i(\omega) + \lambda \right)Z(\omega) \d \P(\omega)
	&=\sum_{i=1}^N x_i \<R_i,Z\> + \lambda \<1,Z\> \\ &= \<(x,\lambda),K^*Z\>
\end{align*}
for all $(x,\lambda)\in\R^N\times\R$ and all $Z\in L^2$  and get
$$ K^*Z = \left(\<R_1,Z\>, \ldots, \<R_N,Z\>, \E\left[Z\right]\right)^T \ \forall Z \in L^2.$$
Further, by considering the convex and closed sets
\begin{align*}
	S &= \left\{x \in \R^N : x^T\mu \geq \expreturn \right\} \\
	T &= \left\{x \in \R^N : x^T \ones^N= 1 \right\},
\end{align*}
the optimization problem \eqref{risk_p3} can be equivalently written as
\begin{align}
\label{risk_p5}
\inf_{(x,\lambda) \in \R^N\times\R} \left\{ \delta_{\R^N_+}(x) + \lambda + \delta_{S\times\R}(x,\lambda) + \delta_{T\times\R}(x,\lambda) + \left (\E\left[ u \right] \circ K \right)(x,\lambda) \right\}.
\end{align}

It is obvious that the functions $(x,\lambda) \mapsto \delta_{\R^N_+}(x) + \lambda$, $\delta_{S\times\R}$ and $\delta_{T\times\R}$ are proper, convex and lower semicontinuous. Furthermore, in the light of Assumption \ref{risk_as1} and by using Fatou' lemma, it follows that $\E\left[ u \right]$ has these properties, as well. This means that problem \eqref{risk_p5} fits into the formulation of the problem $(P)$.

\begin{remark}\label{slater}
For utility functions fulfilling Assumption \ref{risk_as1} and condition \eqref{risk_eqrec}, we have already seen that the optimization problem \eqref{risk_p5} has an optimal solution. Due to the fact that a Slater-type qualification condition is fulfilled, from here it follows (see \cite{Bot10}) that condition \eqref{zin1.2} in Theorem \ref{risk_th0} holds.
\end{remark}

By having a closer look into the formulation of Algorithm \ref{risk_alg1}, one can notice the exposed role played by the proximal points of the functions occurring in the objective of the problem to be solved. Having these determined, one can easily obtain via \eqref{res-indentity} the proximal points of their conjugates, when needed. It is an easy calculation to see that for $(x,\lambda) \in \R^N \times \R$ and $\gamma \in \R_{++}$ it holds
$$\Prox\nolimits_{\gamma f}(x,\lambda) = \argmin_{(y,\nu)\in\R^N_+\times\R}\left\{\gamma \nu + \frac{1}{2}\left\| (y,\nu)-(x,\lambda) \right\|^2 \right\} = \left(\proj_{\R^N_+}\left(x\right),\lambda-\gamma\right),$$
with $$f:\R^N \times \R \rightarrow \overline \R,\ f(y,\nu) = \delta_{\R^N_+}(y) + \nu,$$
$$\Prox\nolimits_{\gamma \delta_{S\times\R}}(x,\lambda) = \left(\proj_{S}\left(x\right),\lambda\right) \ \mbox{and} \  \Prox\nolimits_{\gamma \delta_{T\times\R}}(x,\lambda) = \left(\proj_{T}\left(x\right),\lambda\right),$$
where (see, for instance, \cite[Example 28.16 and Example 3.21]{BauCom11}),
\begin{align*}
	\proj_{S}(x) = \left\{\begin{array}{ll} x,  & \text{if } x^T\mu \geq \expreturn \\ x + \frac{\expreturn-x^T\mu}{\left\|\mu\right\|^2} \mu, &\text{otherwise} \end{array} \right. \quad \mbox{and} \quad
	\proj_{T}(x) = x + \frac{1-x^T\ones^N}{N}\ones^N.
\end{align*}
As we show below, in order to determine the proximal points of $\E\left[ u \right]$ one needs more intricate arguments.
\begin{proposition}
\label{risk_th6.1}
For arbitrary random variables $X\in L^2$ and $\gamma \in \R_{++}$ it holds
\begin{align}
		\Prox\nolimits_{\gamma \E\left[ u \right]}(X)(\omega) = \Prox\nolimits_{\gamma u}\left(X(\omega)\right) \ \forall \omega \in \Omega \text{ a.\,s.}.
\end{align}
\end{proposition}
\begin{proof}
We have
\begin{align*}
	\Prox\nolimits_{\gamma \E\left[ u \right]}(X) &= \argmin_{Y\in L^2}\left\{ \gamma\E\left[u(Y)\right] + \frac{1}{2}\left\| Y-X \right\|^2 \right\} \\
	&= \argmin_{Y\in L^2}\left\{ \gamma\int_{\Omega}u(Y(\omega)) \d\P(\omega) + \frac{1}{2} \int_{\Omega} (Y(\omega)-X(\omega))^2 \d\P(\omega)  \right\} \\
	&= \argmin_{Y\in L^2} \int_{\Omega} \left (\gamma u(Y(\omega)) + \frac{1}{2}(Y(\omega)-X(\omega))^2\right ) \d\P(\omega).
	\end{align*}
Hence, using the interchangeability of integration and minimization (see \cite[Theorem 14.60]{RocWets98}), we have
\begin{align*}
	\Prox\nolimits_{\gamma \E\left[ u \right]}(X)(\omega) = \argmin_{y\in\R}\left\{\gamma u(y) + \frac{1}{2}\left(y-X(\omega)\right)^2 \right\} = \Prox\nolimits_{\gamma u}\left(X(\omega)\right) \ \forall \omega \in \Omega \text{ a.\,s.}.
	\end{align*}
\end{proof}

In what follows we provide explicit formulae for the proximal points of some popular convex utility functions considered the literature, which will be of importance for the numerical experiments presented in the last section and which involve the convex risk measures which rely on them.

\subsection{Piecewise linear utility}
For $\gamma_2<-1<\gamma_1\leq 0$ we consider the piecewise linear utility function
$$u_1:\R \rightarrow \R,\ u_1(t)=\left\{ \begin{array}{ll} \gamma_2t, & \text{if }t\leq0 \\ \gamma_1 t, &\text{if }t>0 \end{array} \right. = \gamma_1 \left[t\right]_{+} - \gamma_2 \left[t\right]_{-}. $$
Assumption \ref{risk_as1} is fulfilled since $u_1(0)=0$ and $-1 \in \partial u_1(0)=\left[\gamma_2,\gamma_1\right]$ and, since for all $d \in \R$ (see \cite{BotFra11})
\begin{equation*}
(u_1)_\infty(d)=\left\{
\begin{array}{ll}
\gamma_2d, & \mbox {if} \  d < 0,\\
0, & \mbox {if} \  d = 0,\\
\gamma_1d, & \mbox {if} \  d > 0,
\end{array}\right.
\end{equation*}
condition \eqref{risk_eqrec} is fulfilled, as well. Hence, $u_1$ gives rise to the lower semicontinuous coherent risk measure
\begin{align}
  \label{risk_eq6}
	\rho_{u_1}(X) = \inf_{\lambda\in\R}\left\{\lambda + \gamma_1\E\left[ X+\lambda\right]_{+} - \gamma_2\E\left[X+\lambda\right]_{-} \right\} \ \forall X \in L^2.
\end{align}
For every $\gamma \in \R_{++}$ and $t \in \R$ it holds
\begin{align*}
	 \Prox\nolimits_{\gamma u_1}\left(t\right) &= \argmin_{s\in\R}\left\{\gamma \left(\gamma_1 \left[s\right]_{+} - \gamma_2 \left[s\right]_{-}\right) + \frac{1}{2}\left(s-t\right)^2 \right\} \\
	&= \left\{\begin{array}{ll} t - \gamma\gamma_2, &\text{if } t<\gamma\gamma_2\\
															0, &\text{if }t \in \left[\gamma\gamma_2,\gamma\gamma_1\right] \\
															t - \gamma\gamma_1, &\text{if }t>\gamma\gamma_1\end{array}\right. \\
	&= \left[t - \gamma\gamma_1\right]_+ - \left[t - \gamma\gamma_2\right]_-.
\end{align*}
When setting $\gamma_1=0$ and $\gamma_2=-\frac{1}{1-\alpha}$ for some $\alpha \in (0,1)$, the convex risk measure \eqref{risk_eq6} becomes the classical so-called \textit{Conditional Value-at-Risk at level $\alpha$} (see, for example, \cite{RocUry00, RocUry02})
\begin{align}
	\label{risk_eq9}
	\cvar\nolimits_{\alpha}:L^2 \rightarrow \R,\ \cvar\nolimits_{\alpha}(X) = \inf_{\lambda \in \R}\left\{\lambda +\frac{1}{1-\alpha}\E\left[X+\lambda\right]_{-} \right\}.
\end{align}
The infimum in the in the expression of the Conditional Value-at-Risk is attained for every $X \in L^2$ at the so-called \textit{Value-at-Risk at level $\alpha$}, i.e.,
$$ \varisk\nolimits_{\alpha}(X) = \argmin_{\lambda\in\R}\left\{\lambda +\frac{1}{1-\alpha}\E\left[X+\lambda\right]_{-} \right\}.$$

\subsection{Exponential utility function}
Consider the exponential utility function $u_2 : \R \rightarrow \R$, $u_2(t) = \exp(-t)-1$. It fulfills Assumption \ref{risk_as1} and, since $(u_2)_{\infty} = \delta_{[0,+\infty)}$, condition \eqref{risk_eqrec} is fulfilled, as well.  It gives rise via \eqref{risk_d1} to the so-called \textit{entropic risk measure}
\begin{align}
  \label{risk_eq7}
	\rho_{u_2}(X) = \inf_{\lambda\in\R}\left\{\lambda + \E\left[ \exp(-X-\lambda) -1 \right] \right\} \ \forall X \in L^2,
\end{align}
which is a lower semicontinuous convex risk measure. For arbitrary $\gamma \in \R_{++}$ and $t \in \R$ it holds
\begin{align*}
	 \Prox\nolimits_{\gamma u_2}\left(t\right) = \argmin_{s\in\R}\left\{\gamma (\exp(-s)-1) + \frac{1}{2}\left(s-t\right)^2 \right\}
\end{align*}
Although no closed form expression for the proximal points of $\gamma u_2$ can be given, these can be efficiently calculated by applying Newton's method under the use of previous iterates as starting points.

\subsection{Indicator utility function}
By choosing the utility function $u_3:\R \rightarrow \oR$, $u_3(t)=\delta_{\left[0,+\infty\right)}(t)$, one has $(u_3)_{\infty} = \delta_{[0,+\infty)}$, thus,  both Assumption \ref{risk_as1} and condition \eqref{risk_eqrec} are fulfilled. It gives rise to the so-called \textit{worst-case risk measure}
\begin{align}
  \label{risk_eq8}
	\rho_{u_3}(X) = \inf_{\substack{\lambda\in\R \\ X+\lambda \geq 0}} \lambda \ = - \essinf X = \esssup(-X) \ \forall X \in L^2,
\end{align}
which is a lower semicontinuous convex risk measure. For arbitrary $\gamma \in \R_{++}$ and $t \in \R$ it holds
\begin{align*}
	 \Prox\nolimits_{\gamma u_3}\left(t\right) = \argmin_{s\in\R}\left\{\gamma \delta_{\left[0,+\infty\right)}(s) + \frac{1}{2}\left(s-t\right)^2 \right\}
	=\proj_{\left[0,+\infty\right)}(t).
\end{align*}

\subsection{Quadratic utility function}
For a fixed $\beta \in \R_{++}$ we consider the quadratic utility function
$$u_4:\R \rightarrow \R,\  u_4(t)= \left\{\begin{array}{ll} \frac{\beta}{2}t^2-t, &\text{if } t\leq \frac{1}{\beta} \\ -\frac{1}{2\beta}, & \text{if }t>\frac{1}{\beta} \end{array}\right..$$
Obviously, $(u_4)_{\infty} = \delta_{[0,+\infty)}$, thus,  both Assumption \ref{risk_as1} and condition \eqref{risk_eqrec} are also fulfilled for this utility function. For arbitrary $\gamma \in \R_{++}$ and $t \in \R$, it holds
\begin{align*}
	 \Prox\nolimits_{\gamma u_4}\left(t\right) = \argmin_{s\in\R}\left\{\gamma u_4(s) + \frac{1}{2}\left(s-t\right)^2 \right\}
	=\left\{\begin{array}{ll} \frac{t+\gamma}{1+\gamma\beta}, &\text{if }t\leq \frac{1}{\beta} \\ t, &\text{if }t>\frac{1}{\beta} \end{array}\right..
\end{align*}

\subsection{Logarithmic utility function}
For $\theta \in \R_{++}$, we consider the logarithmic utility function
$$u_5:\R \rightarrow \oR,\  u_5(t)= \left\{\begin{array}{ll} -\theta \ln \left(1+ \frac{t}{\theta}\right), &\text{if } t> -\theta \\ +\infty, & \text{if }t \leq -\theta \end{array}\right..$$
For this special utility function, one can also show that $(u_5)_{\infty}=\delta_{\left[0,+\infty\right)}$, hence that \eqref{risk_eqrec} is fulfilled. The properties in Assumption \ref{risk_as1} hold as well and therefore, via \eqref{risk_d1}, we obtain the convex risk measure
\begin{align*}
	\rho_{u_5}(X) = \inf_{\substack{\lambda\in\R \\ X+\lambda>-\theta}}\left\{\lambda - \theta \E\left[ \ln\left(1 + \frac{X+\lambda}{\theta}\right) \right] \right\} \ \forall X \in L^2.
\end{align*}
The proximal points of the logarithmic utility function take an explicit expression. For arbitrary $\gamma \in \R_{++}$ and $t \in \R$, it holds
\begin{align*}
	 \Prox\nolimits_{\gamma u_5}\left(t\right) = \argmin_{\substack{s\in\R\\s>-\theta}}\left\{-\gamma\theta\ln\left(1+\frac{s}{\theta}\right) + \frac{1}{2}\left(s-t\right)^2 \right\}
	=\frac{t-\theta}{2} + \sqrt{\frac{(\theta-t)^2}{4}+\theta(\gamma+t)}.
\end{align*}

\section{An alternative approach for CVaR}\label{risk_sectionApproachDual}

In this section we propose an alternative approach for solving the portfolio optimization problem \eqref{risk_p2} when the risk measure in the objective is the Conditional Value-at-Risk at a given confidence level $\alpha \in (0,1)$. To this aim we work in a discrete probability space with $\Omega$ finite, which is the natural framework in real-life applications. We denote by $\left|\Omega\right|$ the \textit{cardinal} of the set $\Omega$.
Thus, the probability measure $\P$ can be represented as a vector $(p_1,...,p_{\left|\Omega\right|}) \in \R^{|\Omega|}$ with $p_i \geq 0, i=1,...,|\Omega|$, and $\sum_{i=1}^{|\Omega|} p_i=1$ and the space of random variables $L^2$ can be identified with the finite-dimensional space $\R^{|\Omega|}$. The investigations made in this section rely on the following dual representation of the Conditional Value-at-Risk given in \cite{LueDoe05,RocUryZab02}, namely, for every $X  \in \R^{|\Omega|}$ it holds
\begin{align}
	\label{risk_def12}
	\cvar\nolimits_{\alpha}(X) = \sup_{q \in \q} -q^TX,
\end{align}
where
\begin{align}
	\label{risk_def13}
	\q = \left\{q \in \R^{\left|\Omega\right|}: \sum_{i=1}^{\left|\Omega\right|}q_i=1,\ 0\leq q_i \leq \frac{p_i}{1-\alpha}, \ i=1,\ldots,\left|\Omega\right|\right\}.
\end{align}

By introducing the convex and closed sets
\begin{align*}
	U &= \left\{x \in \R^{\left|\Omega\right|} : \sum_{i=1}^{\left|\Omega\right|}x_i = 1\right\} \\
	V &= \left\{x \in \R^{\left|\Omega\right|} : 0 \leq x_i \leq \frac{p_i}{1-\alpha},\ i=1,\ldots,\left|\Omega\right|\right\},
\end{align*}
we obtain $\q = U \cap V$, hence, for every $X  \in \R^{|\Omega|}$ it holds
\begin{align}
	\label{risk_def14}
	\cvar\nolimits_{\alpha}(X) = \delta_{\q}^*(-X) = \delta_{U \cap V}^*(-X) = \left(\delta_{U}+\delta_{V}\right)^*(-X)
	= \left(\delta_{U}^* \Box \delta_{V}^*\right)(-X),
\end{align}
where the last equality follows from \cite[Theorem 15.3 \& Proposition 15.5]{BauCom11} and the fact that the intersection of the relative interiors of the sets $U$ and $V$ is nonempty.

Thus, for given $R_i \in \R^{|\Omega|}, i=1,...,N,$ the portfolio optimization problem
\begin{align}
\label{risk_p16}
\inf_{\substack{x^T\mu\geq \expreturn\\x^T\ones^N =1,\ x\in\R^N_+}} \cvar\nolimits_{\alpha}\left(\sum_{i=1}^Nx_iR_i\right)
\end{align}
can be equivalently written as
\begin{align}\label{risk_p17}
\inf_{x\in \R^N} \left\{ \delta_{\R^N_+}(x) + \delta_{S}(x) + \delta_{T}(x) + \left (\delta_{U}^* \Box \delta_{V}^*\right)(Rx)\right\},
\end{align}
where $S$ and $T$ are the sets already introduced in the previous section and $R : \R^N \rightarrow \R^{|\Omega|}$ is defined as $R(x_1,...,x_N) = -\sum_{i=1}^Nx_iR_i$.

One can easily notice that the optimization problem \eqref{risk_p16} fits in the formulation of the general convex optimization problem $(P)$, all the extended real-valued functions present in its objective being proper, convex and lower semicontinuous, and thus it can be solved by means of Algorithm \ref{risk_alg1}. We would also like to point out that the primal-dual splitting algorithms proposed in \cite{ComPes12,Vu11} are also designed to solve convex optimization problems involving infimal convolutions, however, they cannot be applied in this situation. This is because they require that one of the two functions occurring in the infimal convolution are strongly convex, which is for the problem \eqref{risk_p17} not the case. For the implementation of Algorithm \ref{risk_alg1} one has only to determine the projections on some simple convex and closed sets, for which one actually has explicit expressions. We would also like to emphasize that, from this point of view, it is p
 referable to work with the sets $U$ and $V$ separately, instead of dealing with their intersection $\q$.

\begin{remark}\label{remfin}
The approach described above can be analogously employed when considering portfolio optimization problems having as objective a weighted sum of Conditional Value-at-Risk functionals taken at different levels of confidence.
\end{remark}

\section{Numerical experiments}\label{risk_sectionApplications}
\subsection{Simulated data}\label{risk_subsecExperiment1}
The first numerical experiments we made followed the scope to compare different approaches for solving the portfolio optimization problem which assumes the quantification of risk by means of the Conditional Value-at-Risk. More precisely, we compared the performances of Algorithm \ref{risk_alg1} when applied in the context of the approaches proposed in the sections \ref{risk_sectionApproach} and  \ref{risk_sectionApproachDual}, but also with the linear programming approach, widely used in this context in the literature. To this end we used synthetic data obtained by creating random returns $R_i \in \R^{|\Omega|}, i=1,...., N$, where $N$ represents the number of assets in the portfolio.

We first solved with the Matlab routine {\ttfamily linprog} the reformulation of \eqref{risk_p16} as a linear program, that can be easily obtained by means of \eqref{risk_eq9}. Then we used the primal-dual method given in Algorithm \ref{risk_alg1} to solve \eqref{risk_p16} via two different approaches, namely, on the one hand, by solving the reformulation \eqref{risk_p5} proposed in Section \ref{risk_sectionApproach} and, on the other hand, by solving the reformulation \eqref{risk_p17} given in Section \ref{risk_sectionApproachDual} and relying on the dual representation of the objective. We terminated the algorithms when subsequent iterates start to stay within an accuracy level of $1\,\%$ with respect to the set of constraints and to the optimal objective value reported by the linear programming solver. Within these examples we used the confidence level $\alpha=0.95$. The algorithms were implemented in Matlab on an Intel Core i$5$-$2400$ processor under Windows 7 (64 Bit).

\begin{table}[ht]
	\centering
		\begin{tabular}{ r | r || c | c | c }
		$\left|\Omega\right|$   & $N$ & LP ({\ttfamily linprog}) & OCE (iterations) & DR (iterations)  \\ \hline \hline
		$1000$  & $100$       & $1.44$s      & $0.07$s($250$)     &  $0.06$s($247$)   \\
		$1000$  & $500$       & $11.81$s     & $3.45$s($1078$)    &  $1.66$s($519$)   \\
		$1000$  & $1000$      & $29.11$s	   & $12.85$s($1772$)   &  $3.94$s($546$)   \\ \hline
		$10000$ & $100$       & $27.18$s     & $1.89$s($252$)     &  $1.36$s($185$)   \\
		$10000$ & $500$       & $248.79$s    & $38.76$s($1087$)   &  $12.48$s($351$)  \\
		$10000$ & $1000$      & $505.66$s	   & $174.05$s($2465$)  &  $27.79$s($394$)
  \end{tabular}
	\caption{\small CPU times in seconds for solving the portolio optimization problem when using the linear programming (LP) approach, the Optimized Certainty Equivalent (OCE) approach and the dual representation (DR) approach.}
	\label{risk_table:experiment1}
\end{table}

The computational results presented in Table \ref{risk_table:experiment1} show that both approaches proposed in this paper outperformed to linear programming approach from the point of view of the time needed to determine an optimal portfolio for different choices of $N$ and $|\Omega|$. When running Algorithm \ref{risk_alg1} we used the following parameters:
\begin{itemize}
\setlength{\itemsep}{-2pt}
	\item OCE: $\sigma_1=50$, $\sigma_2=50$, $\sigma_3=70/\left\|K\right\|_{L^2}$, $\tau=3/(\sigma_1+\sigma_2+\sigma_3\left\|K\right\|_{L^2}^2)$, $\lambda=1.99$;
	\item DR ($\left|\Omega\right|=1000$): $\sigma_1=2$, $\sigma_2=2$, $\sigma_3=0.1/\left\|R\right\|_{L^2}$, $\tau=2/(\sigma_1+\sigma_2+\sigma_3\left\|R\right\|_{L^2}^2)$, $\lambda=1.99$;
	\item DR ($\left|\Omega\right|=10000$): $\sigma_1=0.1$, $\sigma_2=0.1$, $\sigma_3=0.001/\left\|R\right\|_{L^2}$, $\tau=2/(\sigma_1+\sigma_2+\sigma_3\left\|R\right\|_{L^2}^2)$, $\lambda=1.99$.
\end{itemize}

\subsection{Real data}
For the experiments described as follows we took weekly opening courses over the last 13 years from assets belonging to the indices DAX and NASDAQ in order to obtain the returns $R_i \in \R^{|\Omega|}, i=1,...,N$, for $\left|\Omega\right|=689$ and $N=106$. The data was provided by the Yahoo finance database. Assets which do not support the required historical information like Volkswagen AG (DAX) or Netflix, Inc. (NASDAQ) were not taken into consideration.

We solved the portfolio optimization problem \eqref{risk_p2} by taking as objective function the corresponding convex risk measures induced by the linear, exponential, indicator, quadratic and logarithmic utility function. More precisely, we solved with Algorithm \ref{risk_alg1} its equivalent reformulation \eqref{risk_p2} and used to this end the formulae for the proximal points of each utility function given in Section \ref{risk_sectionApproach}. The values of the expected returns associated with $R_i, i=1,...,N$ ranged from $-0.2690$ (Commerzbank AG, DAX) to $1.4156$ (priceline.com Incorporated, NASDAQ).

Table \ref{risk_table:experiment2} collects some computational results when computing by using  Algorithm \ref{risk_alg1} optimal solutions for  the five utility functions. The stopping criterion was the same as in Section \ref{risk_subsecExperiment1}. Table \ref{risk_table:experiment2} shows that Algorithm \ref{risk_alg1} in combination with the worst-case risk measure, i.e., the one induced by the indicator utility function, performed poorly on the given dataset. It also shows that the algorithm is sensitive with respect to the lower bound of the expected return $\expreturn$. When calculating the proximal points of the exponential utility function we used five iterations of Newton's method with previous iterates as starting points to obtain an appropriate approximation.

\begin{table}[htb]
	\centering
		\begin{tabular}{ c || l | l | l | l | l }
		$\expreturn$   & linear ($\alpha=0.95$) & exponential & indicator & quadr. ($\beta=1$) & log. ($\theta=5$)  \\ \hline \hline
		$0.3$  & $0.14$s($500$)       & $0.18$s($402$)     & - ($>15000$)   &  $0.05$s($170$)  & $0.53$s($1891$) \\
		$0.5$  & $0.15$s($520$)       & $0.15$s($336$)     & - ($>15000$)   &  $0.06$s($196$)  & $0.38$s($1335$) \\
		$0.7$  & $0.33$s($1202$)      & $0.31$s($682$)	   & - ($>15000$)   &  $0.06$s($186$)  & $0.72$s($2570$) \\
		$0.9$  & $0.32$s($1164$)      & $0.40$s($885$)     & - ($>15000$)   &  $0.08$s($272$)  & $1.07$s($3820$) \\
		$1.1$  & $0.41$s($1526$)      & $6.80$s($15222$)   & - ($>15000$)   &  $0.14$s($486$)  & $1.18$s($4198$) \\
		$1.3$  & $0.42$s($1570$)      & $5.45$s($12155$)	 & - ($>15000$)   &  $0.41$s($1476$) & $6.61$s($23547$)
  \end{tabular}
	\caption{\small CPU times in seconds and the number of iterations when solving the portfolio optimization problem \eqref{risk_p2} under different utility functions.}
	\label{risk_table:experiment2}
\end{table}

Figure \ref{fig:risk_efficient_frontier} shows the efficient frontiers for problem \eqref{risk_p2}. When using the indicator utility function we stopped the algorithm after a number of $30000$ iterations. The objective value, however, still oscillated in this scenario.

\begin{figure}[htb]	
	\centering
	\captionsetup[subfigure]{position=top}
	\subfloat[linear utility]{\includegraphics*[width=0.48\textwidth]{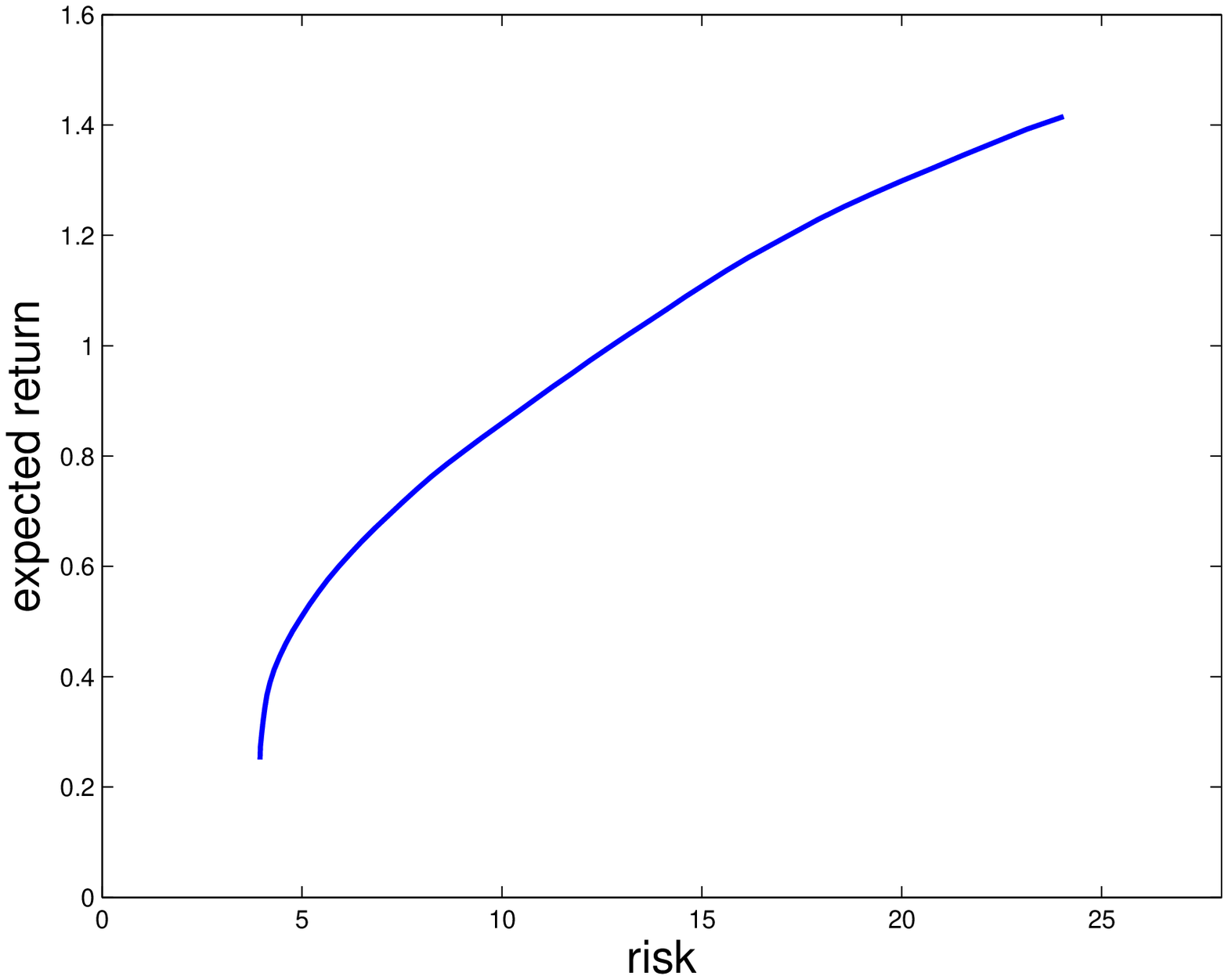}} \hspace{2mm}
	\subfloat[exponential utility]{\includegraphics*[width=0.48\textwidth]{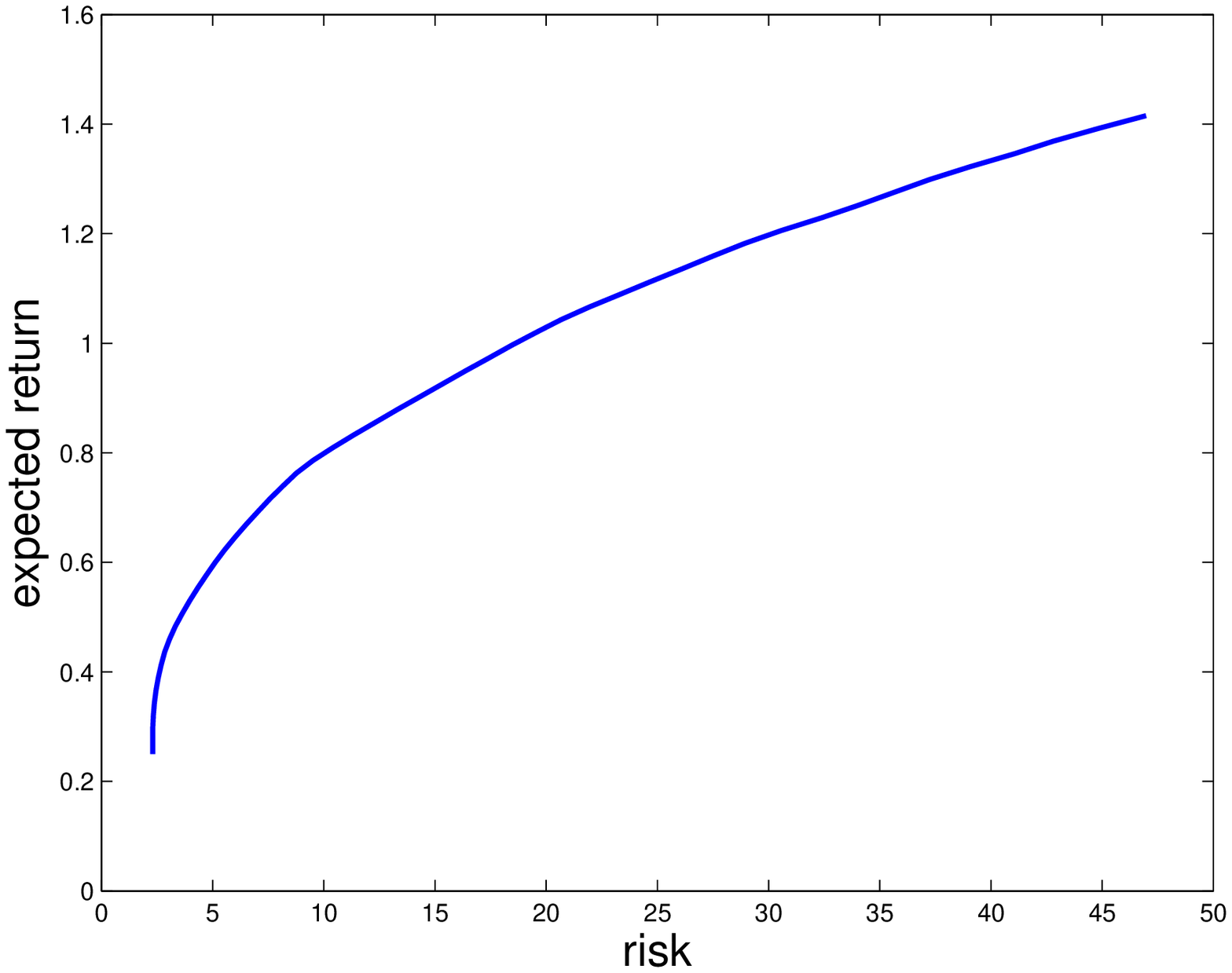}}
	
	\subfloat[indicator utility]{\includegraphics*[width=0.48\textwidth]{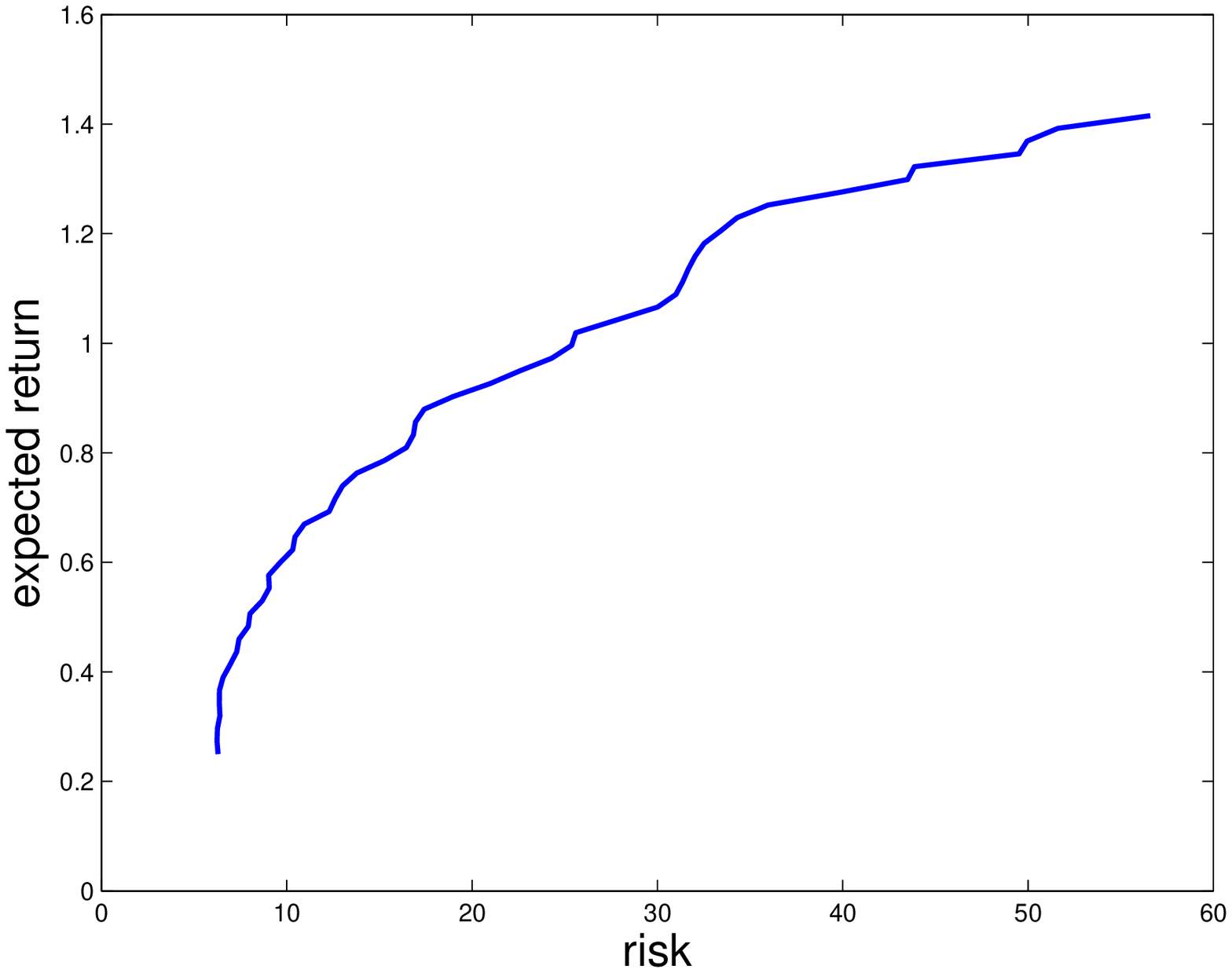}} \hspace{2mm}
	\subfloat[quadratic utility]{\includegraphics*[width=0.48\textwidth]{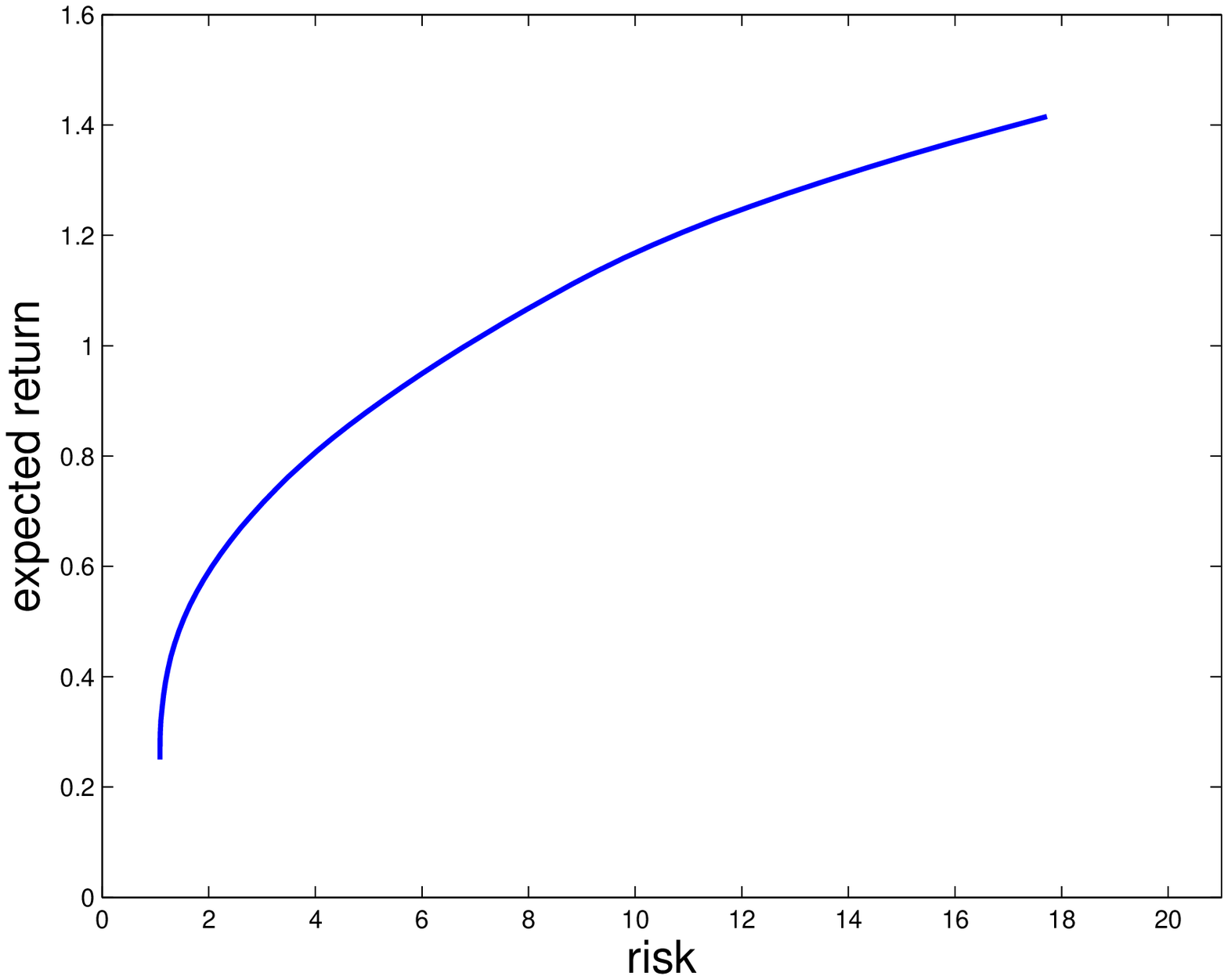}}
	
	\subfloat[logarithmic utility]{\includegraphics*[width=0.48\textwidth]{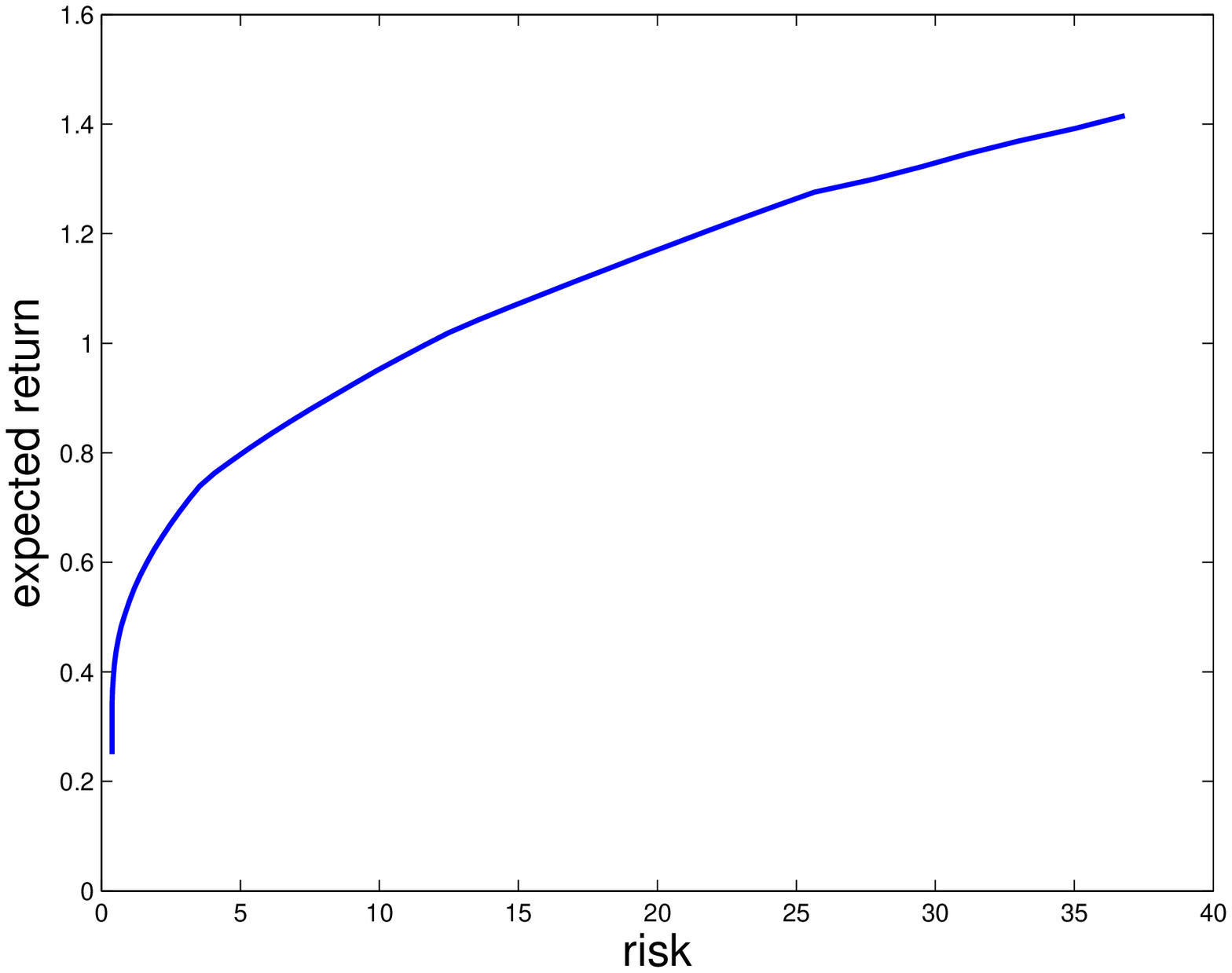}}
	\caption{\small The efficient frontiers for the portfolio optimization problem under different convex risk measurements.}
	\label{fig:risk_efficient_frontier}	
\end{figure}

\end{document}